\documentclass[a4paper, 10pt]{article}

\usepackage{packageList}
\usepackage{macros}
\usepgfplotslibrary{fillbetween}

\graphicspath{{./Figures/}}
\pgfplotsset{compat=1.11}
\newlength\fwidth

\theoremstyle{remark}
\newtheorem{remark}[theorem]{Remark}

\title{Reliable sampling-based RKHS norm estimation via superconvergence}
\author[1,2]{Tizian Wenzel \thanks{\href{mailto:wenzel@math.lmu.de}{wenzel@math.lmu.de}}}
\author[3]{Abdullah Tokmak \thanks{\href{mailto:abdullah.tokmak@aalto.fi}{abdullah.tokmak@aalto.fi}}}
\author[2,4]{Christian Fiedler \thanks{\href{mailto:christian.fiedler@tum.de}{christian.fiedler@tum.de}}}

\affil[1]{Ludwig Maximilian University of Munich (Munich, Germany)}
\affil[2]{Munich Center for Machine Learning (Munich, Germany)}
\affil[3]{Department of Electrical Engineering and Automation, Aalto University (Espoo, Finland)}
\affil[4]{Department of Mathematics, Technical University of Munich (Munich, Germany)}

\newcommand{\fakepar}[1]{\vspace{1mm}\noindent\emph{#1.}}

\begin{document}

\maketitle

\begin{abstract}
Kernel methods are one of the cornerstones of learning-based control, modern system identification, surrogate modelling, and related fields.
A key advantage of this class of learning and function approximation methods is the availability of quantitative error bounds, which in turn play a key role in guaranteeing the safety of learned controllers and related learning-based algorithms.
However, these error bounds rely on a particular property of the target function---its reproducing kernel Hilbert space (RKHS) norm---which is usually impossible to obtain in practice.
Motivated by this severe shortcoming, we present a novel sampling-based RKHS norm estimation approach with a solid theoretical foundation, leveraging very recent advances in the theory of superconvergence in kernel methods.
Our method is applicable to a broad range of practically relevant function classes and requires only reasonable prior knowledge about the target function.
Extensive numerical experiments demonstrate the efficacy and practical applicability of the proposed method.
By providing a reliable RKHS norm estimation approach, we remove a major obstacle to the practical deployment of learning-based control algorithms.
\end{abstract}

\section{Introduction}
Kernel methods are a popular and versatile class of learning \cite{steinwart2008support} and approximation approaches \cite{fasshauer2015kernel}, 
with numerous applications in learning-based control \cite{hewing2020learning}, 
system identification \cite{pillonetto2014kernel}, 
surrogate modelling \cite{gramacy2020surrogates}, and related fields.
These algorithms come with efficient and reliable numerical algorithms,
the ability to systematically include prior knowledge,
and an extensive and well-understood theory \cite{steinwart2008support,wendland2004scattered,fasshauer2015kernel},
including quantitative error bounds.
It is the latter aspect that makes kernel methods particularly attractive for safety-critical applications like learning-based control \cite{hewing2020learning},
approximate model predictive control \cite{tokmak2025automatic}, or safe Bayesian optimisation (BO) \cite{sui2015safe}.
For example, in learning-based control with safety guarantees, typically a machine learning algorithm is used that can quantify the remaining uncertainty about the underlying system, which in turn is then taken care of by a robust control method.
Similarly, quantitative error bounds are used in safe BO to ensure that the iterative optimisation process only queries safe inputs with high probability.
Suitable error bounds are available for kernel methods, and most importantly, these can be numerically evaluated \cite{fiedler2021learning,fasshauer2015kernel}.
More precisely, in applications like learning-based control, it is primarily frequentist and worst-case uncertainty bounds that are relevant, cf. e.g. the discussion in \cite{fiedler2021learning},
and such bounds are available inter alia for noise-free kernel interpolation as well as kernel ridge regression and Gaussian process regression.
However, to numerically evaluate these bounds, a concrete bound on the reproducing kernel Hilbert space (RKHS) norm of the target function is required,
and in practice this is rarely available.
In turn, this poses a severe obstacle to the practical availability of methods utilising such bounds, since in applications like learning-based control and safe BO all safety guarantees hinge on quantitative uncertainty bounds.
In fact, since in general no upper bound on the RKHS norm of the target function can be derived from reasonable prior knowledge, heuristics are used in practice that invalidate all safety guarantees, cf. the extensive discussion in \cite{fiedler2024safety}.

The relevance and significance of this issue is by now well-known in learning-based control, though no general solution exists so far.
One way to overcome this problem is to avoid uncertainty bounds relying on the RKHS norm altogether.
For example, one can combine kernel methods with bounds based on geometric properties like Lipschitz continuity \cite{fiedler2022learning,fiedler2024safety}, however, this has to be done on a case-by-case basis and one loses a lot of the available theory for kernel methods.
Another approach is to work in a probabilistic setting, where the underlying function is assumed to be random, see e.g.\ \cite{lederer2019uniform,tokmak2024pacsbo, tokmak2025safe}.
For many applications in learning-based control, this is not suitable either, since a frequentist framework with a fixed, but unknown ground truth function is often more appropriate than a probabilistic setup with a random target function.
Finally, a more recent direction is to try to estimate the RKHS norm of a target function from samples.
While there have been various proposals how to do this, see the references below, to the best of our knowledge there does not exist a reliable method with a solid theoretical foundation yet.
In this work, we provide a significant step forward on this problem. %

We consider the setting of having noise-free and safe access to the target function, a setup that appears, for example, in approximate (robust) model predictive control \cite{tokmak2025automatic}, and we want to estimate the RKHS norm of this target function from samples. %
As is well-known, in general it is only possible to get an \emph{underestimation} of the true RKHS norm from finite data, cf.\ e.g.\ \cite[Section~A]{scharnhorst2022robust}.
More precisely, the standard kernel interpolator is a minimum norm interpolator, and hence its computable norm is a lower bound on the norm of the target function.
To get an estimate of the RKHS norm or at least an upper bound thereof, heuristics have been proposed that utilise the fact that the norm of the minimum norm kernel interpolator increases with an increasing number of data points.
To the best of our knowledge, \cite{hashimoto2022learning} was the first work proposing to use this to get an RKHS norm estimate, albeit without a systematic method to perform this estimate.
Subsequently, \cite{tokmak2025automatic} introduced a simple, quantitative heuristic based on a postulated exponential increase of the norm of the interpolator.
In contrast, we present the first method with a rigorous theoretical foundation,
building on recent advances in kernel approximation theory.
We use the fact that for certain common kernels, the smoothness of a target function is exactly characterised by the convergence rate of the kernel approximation error, which follows from very recent sharp direct and inverse results \cite{wenzel2026sharp}.
In turn, from the observed convergence behaviour of the approximation error we can then estimate the RKHS norm or even an upper bound thereof.
For this, we only need to assume sufficient smoothness of the target function, without even knowing the exact quantitative smoothness.

\emph{Outline}. In Section \ref{sec:background}, we provide technical background on RKHSs and kernel interpolation.
We derive and describe our method in Section \ref{sec:rkhs_norm_estimation}, focusing on finitely smooth  Sobolev kernels. %
We evaluate our method with extensive numerical experiments in Section \ref{sec:experiments},
and we conclude with Section \ref{sec:conclusion}.

\section{Background} \label{sec:background}
We start with some background on kernels and kernel interpolation, and refer to \cite{wendland2004scattered,steinwart2008support} for more details.

\subsection{Kernels, RKHSs and interpolation}
\fakepar{Kernels and RKHSs}
We will consider strictly positive definite kernels $k$ on bounded Lipschitz domains $\Omega \subset \R^d$,
i.e., symmetric functions $k: \Omega \times \Omega \rightarrow \R$ 
such that the kernel matrix $A_X := (k(x_i, x_j))_{1 \leq i, j \leq |X|} \in \R^{|X| \times |X|}$ 
is positive definite for any choice of pairwise distinct points $X \subset \Omega$.
For such a kernel $k$, there is a unique Hilbert space of functions $\ns$ for which $k$ is a reproducing kernel,
i.e., $\langle f, k(\cdot, x) \rangle_{\ns} = f(x)$ for all $f \in \ns$ and $x \in \Omega$,
and we call $\ns$ a reproducing kernel Hilbert space (RKHS).
As a consequence of the reproducing property $k$, the RKHS norm of a function of the form $g(x)=\sum_{i=1}^n \alpha_i k(x,x_i)$, where $\alpha_1,\ldots,\alpha_n \in \R$ and $x_1,\ldots,x_n\in\Omega$, can be computed in closed form as
\begin{align}
\label{eq:rkhs_norm_finite}
    \|g\|_{\ns} = \sqrt{\sum_{i,j=1}^n \alpha_i \alpha_j k(x_j,x_i)}.
\end{align}
Note that a general $f\in \ns$ does not have this form and hence its RKHS norm cannot be computed in this way.

\fakepar{Kernel interpolation} 
Due to the strict positive definiteness of $k$,
for any function $f: \Omega\rightarrow\R$ and any non-empty finite subset $X\subseteq \Omega$,
there exists a unique interpolating function $s_{f,X}$ from $\ns$ with minimum norm,
i.e., $s_{f,X}(x)=f(x)$ for all $x\in X$, and among all functions from $\ns$ with this property, $s_{f,X}$ has the minimum norm.
In addition, it holds that $s_{f,X} \in \Sp \{ k(\cdot, x), x \in X \}$,
and if $f \in \ns$,
then the interpolant $s_{f, X}$ is in fact the orthogonal projection of $f$ onto the subspace $\Sp \{ k(\cdot, x), x \in X \}$.
Furthermore, we have the important property
\begin{align}
\label{eq:boundedness_rkhs_norms}
f \in \ns ~ \Leftrightarrow ~ \exists C > 0 ~ \forall X \subset \Omega ~ \Vert s_{f, X} \Vert_{\ns} \leq C,
\end{align}
and in fact one can choose $C = \Vert f \Vert_{\ns}$.
Thus, for $f \in \ns$ and the fact that $s_{f,X}$ is then an orthogonal projection, 
one has $\Vert s_{f, X} \Vert_{\ns} \leq \Vert f \Vert_{\ns}$,
which always provides a lower bound for the RKHS norm of $f$. 
The goal of this work is to derive a method to reliably estimate also an upper bound for the RKHS norm of $f$.

\fakepar{Error bounds} An upper bound on the RKHS norm is important since it allows to derive quantitative error bounds.
Focusing on the interpolation setting introduced above with $f\in\ns$, we have for all $x\in \Omega$ that
\begin{align} \label{eq:errorBoundWithPowerFunction}
    |f(x) - s_{f,X}(x)| & \leq P_X(x)\sqrt{\|f\|_{\ns}^2 - \|s_{f,X}\|_{\ns}^2} \notag \\
        & \leq P_X(x)\|f\|_{\ns},
\end{align}
where $P_X(x)$ is the power function, which depends only on $X$, $k$, and $x$, and can be explicitly computed \cite{fasshauer2015kernel}.
In particular, this allows one to get error bounds everywhere on the input domain despite having only finitely many data points,
and these bounds are in general sharp, cf. Figure \ref{fig:vis_power_func} for an illustration.
Alternatively, one can also derive error bounds that depend on geometric properties of the data set $X$,
in particular, the fill distance $h_X := \sup_{x \in \Omega} \min_{x_i \in X} \Vert x - x_i \Vert$
and the separation distance $q_X := \min_{x_i \neq x_j \in X} \Vert x_i - x_j \Vert$,
as well as the uniformity constant $\rho_X := h_X/q_X \geq 1$ describing how uniform the points $X$ are distributed.
Furthermore, these quantities can be used to describe the convergence of $s_{f,X}$ to $f$ in various norms.
\emph{Direct results} provide upper bounds on the approximation errors,
whereas \emph{inverse results} allow to infer the level of smoothness of the target function from the convergence of the approximation error.
For many common kernels there are now \emph{sharp direct and inverse results} available, 
which in turn describe the smoothness of a target function precisely through the convergence of the approximation error,
cf. Theorem \ref{th:one_to_one_rkhs} for such a result.

\fakepar{Power spaces} In order to obtain sharp error estimates \cite{wenzel2026sharp},
it is necessary to take into account the exact smoothness of the target function $f$ of interest,
measured in terms of so-called power space smoothness \cite{steinwart2012mercer}.
For this, assume from now on that $k$ is continuous, and define the kernel integral operator
\begin{align*}
T_k: L_2(\Omega) \rightarrow L_2(\Omega), ~
f \mapsto \int_\Omega k(x, z)f(z) ~ \mathrm{d}z,
\end{align*}
which is a compact self-adjoint operator.
Mercer's theorem yields an orthonormal basis of eigenfunctions $(\varphi_j)_{j \in \N} \subset L_2(\Omega)$ of $T_k$ with corresponding ordered eigenvalues $(\lambda_j)_{j \in \N} \subset \R_+$.
These Mercer eigenfunctions and -values can be used to define power spaces for $\vartheta \geq 0$ as
\begin{align}
\label{eq:power_space}
(\mathcal{H}_{k}(\Omega))_\vartheta := \left\{ f \in L_2(\Omega) ~:~ \sum_{j=1}^\infty \frac{|\langle f, \varphi_{j} 
\rangle_{L_2(\Omega)}|^2}{\lambda_{j}^\vartheta} < \infty \right\}.
\end{align}
The regime $\vartheta \in [0, 1]$ is usually called the \emph{escaping the native space}\footnote{In approximation theory, an RKHS is also called a \emph{native space}. The regime $\vartheta\in[0,1]$ means that the target function is considered outside of the RKHS, hence the terminology.} or the \emph{misspecified} regime \cite{narcowich2006sobolev},
while the regime $\vartheta \in [1, 2]$ is called the \emph{superconvergence} regime \cite{schaback2018superconvergence,wenzel2026sharp}.
The latter terminology is motivated by the fact that additional smoothness (over the smoothness of the RKHS) leads to faster convergence rates.
Note that for $\vartheta > 2$, a saturation effect occurs which limits the convergence rates \cite{wenzel2026optimal}.
\subsection{Finitely smooth kernels} \label{subsec:background_finitely}
First, we focus on finitely smooth Sobolev kernels, i.e., kernels with RKHSs that are norm-equivalent to the standard Sobolev space of smoothness $\tau > d/2$, so $\ns \asymp H^\tau(\Omega)$.
\begin{assumption}
\label{ass:kernel_domain}
Let $\Omega \subset \R^d$ be a compact Lipschitz region and let $k: \Omega \times \Omega \rightarrow \R$ be a continuous kernel such that $\ns \asymp 
H^\tau(\Omega)$ for some $\tau > d/2$. 
\end{assumption}
For such kernels, the power space $(\ns)_\vartheta$ can be related to the Sobolev space $H^{\vartheta \tau}(\Omega)$,
and an overview of the equivalence and subset relations is visualized in \Cref{fig:visualization}.
For such finitely smooth kernels,
recently a one-to-one correspondence between smoothness (measured in terms of the power spaces) and approximation rate of kernel methods has been established \cite{wenzel2026sharp,wenzel2026optimal}:
A function $f \in \mathcal{C}(\Omega)$ has power space smoothness $\vartheta_0 \in (0, 2]$ (meaning that $f \in (\ns)_{\vartheta'}$ for all $\vartheta' < \vartheta_0$ and there no larger $\vartheta_0' > \vartheta_0$ such that this holds),
if and only if $f$ can be approximated with a rate of $\vartheta_0 \tau$ (meaning that for all $\vartheta' < \vartheta_0$ there exists a constant $C$ such that $\Vert f - s_{f, X} \Vert_{L_2(\Omega)} < C h_X^{\vartheta' \tau}$ for quasi-uniform $X \subset \Omega$ and there exists no larger $\vartheta_0' > \vartheta_0$ such that this holds).
These results were formulated using $L_2(\Omega)$ norms, and in \Cref{subsec:rkhs_finitely_smooth} we will extend them to RKHS norms,
which can then be used to estimate the RKHS norm of a function.
\subsection{Infinitely smooth kernels} \label{subsec:background_gaussian}

Beyond finitely smooth kernels like Matérn kernels, infinitely smooth kernels  are also frequently used in applications.
A very popular example is the Gaussian kernel $k(x, z) = \exp(-\varepsilon^2 \Vert x - z \Vert^2)$,
where $\varepsilon > 0$ denotes a shape parameter.
In this work, we do not consider the Gaussian kernel, primarily for two reasons.

First, the RKHS of this kernel is rather small \cite{steinwart2006explicit}, which is problematic from a practical point of view.
Furthermore, since the corresponding power spaces are also rather small, it is not easy to decide when a given function is included in such a power space in the first place.
Second, the approximation-theoretic properties of sampling-based approximation using the Gaussian kernel are not as well developed:
While available direct statements usually provide exponential rates of convergence \cite{rieger2010sampling},
and general superconvergence theory is still applicable \cite{karvonen2025general},
so far no sharp inverse estimates have been derived yet \cite{wenzel2025sharp}.
This might be due to a mismatch between these convergence rates and the corresponding stability estimates, see \cite[Theorem 11.22]{wendland2004scattered} vs.\ \cite[Corollary 12.4]{wendland2004scattered}.
One important aspect in this context is the observation that the optimal point distribution for interpolation does not seem to be uniform anymore, since oversampling near the boundary of the inputs domain appears to be beneficial \cite{rieger2014improved}.

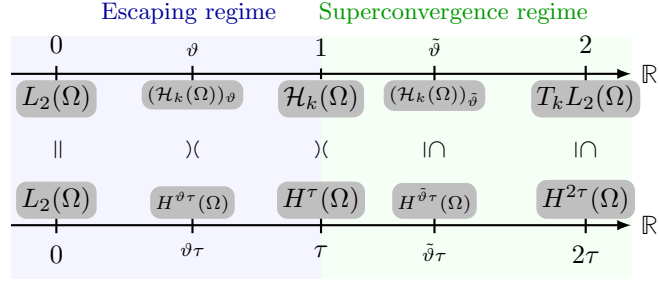
\begin{figure}[t]
\setlength\fwidth{.8\textwidth}
\begin{center}
\begin{tikzpicture}[>=latex, thick]

\begin{scope}[on background layer]
  \fill[blue!10, opacity=0.4]
    (.9,-2.7cm) rectangle (5.0cm,0.5cm);

  \fill[green!10, opacity=0.4]
    (5.0cm,-2.7cm) rectangle (9.1cm,0.5cm);
\end{scope}

\node[blue!60!black] at (3.25cm,0.8cm)
  {\small Escaping regime};

\node[green!60!black] at (6.75cm,0.8cm)
  {\small Superconvergence regime};

\draw[->] (.9,0) -- (9.1cm,0) node [right] {$\R$};

\draw (1.5,-3pt) -- (1.5,3pt) node[above=1pt]{0};
\draw (1.5,-2pt) node[below=0pt, align=center, fill=lightgray, rounded corners, inner sep=2pt]{$L_2(\Omega)$};

\draw (3.3,-3pt) -- (3.3,3pt) node[above=0pt]{\scriptsize $\vartheta$};
\draw (3.3, -2pt) node[below=0pt, align=center, fill=lightgray, rounded corners, inner sep=2pt]{\scriptsize $(\mathcal{H}_{k}(\Omega))_{\vartheta}$};

\draw (5.0,-2pt) -- (5.0,2pt) node[above=1pt]{1};
\draw (5.0, -2pt) node[below=0pt, align=center, fill=lightgray, rounded corners, inner sep=2pt]{$\mathcal{H}_{k}(\Omega)$};

\draw (6.5,-3pt) -- (6.5,3pt) node[above=0pt]{\scriptsize $\tilde{\vartheta}$};
\draw (6.5, -2pt) node[below=0pt, align=center, fill=lightgray, rounded corners, inner sep=2pt]{\scriptsize $(\mathcal{H}_{k}(\Omega))_{\tilde{\vartheta}}$};

\draw (8.5,-3pt) -- (8.5,3pt) node[above=1pt]{2};
\draw (8.5, -2pt) node[below=0pt, align=center, fill=lightgray, rounded corners, inner sep=2pt]{$T_{k}L_2(\Omega)$}; 

\draw[->] (.9,-2cm) -- (9.1cm,-2cm) node [right] {$\R$};

\draw (1.5,-2cm+3pt) -- (1.5,-2cm-3pt) node[below=1pt]{0};
\draw (1.5,-2cm+2pt) node[above=0pt, align=center, fill=lightgray, rounded corners, inner sep=2pt]{$L_2(\Omega)$};

\draw (3.3,-2cm+3pt) -- (3.3,-2cm-3pt) node[below=0pt]{\scriptsize $\vartheta \tau$};
\draw (3.3, -2cm+2pt) node[above=0pt, align=center, fill=lightgray, rounded corners, inner sep=2pt]{\scriptsize $H^{\vartheta \tau}(\Omega)$};

\draw (5.0,-2cm+3pt) -- (5.0,-2cm-3pt) node[below=1pt]{$\tau$};
\draw (5.0, -2cm+2pt) node[above=0pt, align=center, fill=lightgray, rounded corners, inner sep=2pt]{$H^\tau(\Omega)$};

\draw (6.5,-2cm+3pt) -- (6.5,-2cm-3pt) node[below=0pt]{\scriptsize $\tilde{\vartheta}\tau$};
\draw (6.5, -2cm+2pt) node[above=0pt, align=center, fill=lightgray, rounded corners, inner sep=2pt]{\scriptsize $H^{\tilde{\vartheta} \tau}(\Omega)$};

\draw (8.5,-2cm+3pt) -- (8.5,-2cm-3pt) node[below=1pt]{$2\tau$};
\draw (8.5, -2cm+2pt) node[above=0pt, align=center, fill=lightgray, rounded corners, inner sep=2pt]{$H^{2\tau}(\Omega)$};

\draw (1.5, -1cm) node[align=center]{\rotatebox[origin=c]{270}{$=$}};
\draw (3.3, -1cm) node[align=center]{\rotatebox[origin=c]{270}{$\asymp$}};
\draw (5.0, -1cm) node[align=center]{\rotatebox[origin=c]{270}{$\asymp$}};
\draw (6.5, -1cm) node[align=center]{\rotatebox[origin=c]{270}{$\subseteq$}};
\draw (8.5, -1cm) node[align=center]{\rotatebox[origin=c]{270}{$\subseteq$}};

\end{tikzpicture}
\end{center}
\caption{Visualization of the scale of power spaces (top arrow) and Sobolev spaces (bottom arrow).
Several special cases like $L_2(\Omega)$, $\ns \asymp H^{\tau}(\Omega)$ and $T_k L_2(\Omega) \subset H^{2\tau}(\Omega)$ are depicted,
as well as the equivalence and subset relations.
}
\label{fig:visualization}
\end{figure}
\section{RKHS norm estimates via superconvergence} \label{sec:rkhs_norm_estimation}
We now present the proposed approach to estimate RKHS norms.
In \Cref{subsec:rkhs_finitely_smooth}, %
we develop the necessary theory for finitely smooth kernels,
which allows us to derive two algorithms in \Cref{subsec:alg} for estimating the RKHS norm and an upper bound on the RKHS norm, respectively.
\subsection{Theoretical foundations} \label{subsec:rkhs_finitely_smooth}
We consider a strictly positive definite kernel $k$ such that $\ns \asymp H^\tau(\Omega)$, cf. Assumption \ref{ass:kernel_domain},
which covers plenty of frequently used kernels in practice.
The following theorem extends the one-to-one correspondence of \cite{wenzel2026sharp,wenzel2026optimal} from the $L_2(\Omega)$-norm to the $\ns$-norm.
\begin{theorem}
\label{th:one_to_one_rkhs}
Under Assumption \ref{ass:kernel_domain}, it holds for $\vartheta \in (1, 2]$:
\begin{align*}
&~ \forall \vartheta' < \vartheta ~ \quad f \in (\ns)_{\vartheta'}  \quad \Leftrightarrow \\
~&~ \forall \vartheta' < \vartheta ~ \exists C>0 ~ \forall X \subset \Omega: ~ \Vert f - s_{f, X} \Vert_{\ns} \leq C h_X^{(\vartheta'-1)\tau}
\end{align*}
\end{theorem}
\begin{proof}
``$\Rightarrow$'' This follows immediately from \cite[Corollary 16]{karvonen2025general}.
``$\Leftarrow$'' 
Sampling inequalities \cite{narcowich2006sobolev} yield $\Vert f - s_{f, X} \Vert_{L_2(\Omega)} \leq C' h_X^\tau \Vert f - s_{f, X} \Vert_{\ns}$,
such that we obtain the estimate $\Vert f - s_{f, X} \Vert_{L_2(\Omega)} \leq CC' h_X^{\vartheta'\tau}$.
Thus, \cite[Theorem 1]{wenzel2026sharp} is applicable and yields the statement.
\qed
\end{proof}
\begin{remark}
\label{rem:power_spaces}
The power spaces $(\ns)_{\vartheta}$ for $\vartheta > 1$,
can be difficult to characterise.
There are some Sobolev kernels, for example Matérn kernels, that allow a description in terms of Sobolev smoothness $\vartheta \tau$ and additional boundary conditions,
see \cite[Section 6]{karvonen2025general} or \cite[Proposition 9]{karvonen2026piecewise} for an explicit example,
and finding such boundary conditions is in general very hard.
However, this phenomenon is in general not a problem for our approach.
These boundary conditions usually do not occur immediately as soon as $\vartheta > 1$, but only for some larger $\vartheta$,
and hence we still obtain some superconvergence.
In the literature, a superconvergence rate of $h^{1/2}$ was frequently observed \cite{schaback2018superconvergence} and examined for instance in \cite[Proposition 9]{karvonen2026piecewise},
even without matching any particular boundary conditions.
As a consequence, our methods are still applicable for such kernels, even without matching any boundary condition.
\end{remark}
Throughout this section we assume that $f \in (\ns)_\vartheta$ for some $\vartheta > 1$, so in particular, $f \in \ns$.
While in some applications it might be unclear whether this holds, 
we stress that $f \notin \ns$ leads to unbounded growth of the RKHS norms of the interpolants, cf.\ Eq.~\eqref{eq:boundedness_rkhs_norms},
and this can be easily detected.
We remark that this growth can even be characterized accurately, 
see \cite[Theorem 3.3,~3.6]{avesani2025sobolev}. %
\subsection{Algorithms} \label{subsec:alg}
Assume now $f \in (\ns)_\vartheta \subset \ns$ for some (potentially unknown) $\vartheta \in (1, 2]$.
The preceding result \Cref{th:one_to_one_rkhs} shows that the RKHS norm of the residual $f - s_{f, X}$ decays with an algebraic rate of convergence w.r.t.\ the fill distance $h_X$.
We leverage this theoretical result to propose two algorithms for estimating the RKHS norm of the unknown function $f$.
Both algorithms approximate the function $f$ using a sequence of increasingly denser sets of (finitely many) points,
computing and tracking the RKHS norms of the resulting interpolants.
In \Cref{fig:vis_alg}, the convergence behaviour of these interpolants is illustrated for the two algorithms described in the following. %
\subsubsection{Algorithm 1}
We first describe an algorithm that aims at estimating the RKHS norm of a given target function.
Using the well-known orthogonality property of kernel interpolation, i.e., $f - s_{f, X} \perp s_{f, X}$, we have that
\begin{align*}
    \Vert f \Vert_{\ns}^2 &= \Vert f - s_{f, X} \Vert_{\ns}^2 + \Vert s_{f, X} \Vert_{\ns}^2.
\end{align*}
The term $\Vert s_{f, X} \Vert_{\ns}^2$ can be computed from $s_{f,X}$,
see \eqref{eq:rkhs_norm_finite},
and we have a precise description of the decay behaviour of $\Vert f - s_{f, X} \Vert_{\ns}^2$, provided by \Cref{th:one_to_one_rkhs}.
Therefore, we propose to perform a least square fit of the model $c_1 - c_1' h^{\beta_1}$ (with parameters $c_1,c_1',\beta_1$) to the data points  $(h_{X_i}, \Vert s_{f, X_i} \Vert_{\ns}^2)_{i=0, ..., n_0}$, where $X_0,\ldots,X_{n_0}$ is a sequence of quasi-uniform points with decreasing fill distances.
Finally, $\sqrt{c_1}$ can be used as an estimate of $\Vert f\Vert_{\ns}$.
The resulting method is summarised in \Cref{alg:alg1}.
\begin{algorithm}
\caption{RKHS norm estimation}
\label{alg:alg1}
\begin{algorithmic}[1]
\Require Input points $(X_i)_{i=0, ..., n_0} \subset \Omega$, corresponding function values of $f$
\Ensure RKHS norm prediction
\For{$i=0, ..., n_0$}
    \State Compute and store $h_{X_i}$, $\Vert s_{f, X_i} \Vert_{\ns}^2$
\EndFor
\State Fit model $h \mapsto c_1 - c_1' h^{\beta_1}$ (determine $c_1,c_1',\beta_1$) to $(h_{X_i}, \Vert s_{f, X_i} \Vert_{\ns}^2)_{i=0, ..., n_0}$ \\
\Return RKHS estimate $\sqrt{c_1}$
\end{algorithmic}
\end{algorithm}
\subsubsection{Algorithm 2}
Second, we now derive an algorithm for estimating an upper bound on the RKHS norm of a target function.
Consider a nested sequence of quasi-uniform points $(X_n)_n$.
Using the triangle inequality and then \Cref{th:one_to_one_rkhs} leads to
\begin{align*}
    & \Vert s_{f, X_{n+1}} - s_{f, X_n} \Vert_{\ns} \\
        & \hspace{0.5cm} \leq \Vert f - s_{f, X_{n+1}} \Vert_{\ns} + \Vert f - s_{f, X_n} \Vert_{\ns} \\
        & \hspace{0.5cm} \leq C h_{X_{n+1}}^{(\vartheta'-1)\tau} + C h_{X_n}^{(\vartheta'-1)\tau} \\
        & \hspace{0.5cm} \leq 2C h_{X_n}^{(\vartheta'-1)\tau}
\end{align*}
Due to the nestedness of the point sets, we can repeatedly leverage the orthogonality property of kernel interpolation.
Together with the previous estimate, this results in 
\begin{align}
\label{eq:estimate_via_geom_sum}
    & \Vert f \Vert_{\ns}^2
    = \lim_{n \rightarrow \infty} \Vert s_{f, X_n} \Vert_{\ns}^2 \notag \\
    & \hspace{0.5cm} = \lim_{n \rightarrow \infty} \Vert s_{f, X_n} - s_{f, X_{n-1}} \Vert_{\ns}^2 + \Vert s_{f, X_{n-1}} \Vert_{\ns}^2 \notag \\
    & \hspace{0.5cm} = \lim_{n \rightarrow \infty} \sum_{j=1}^n \Vert s_{f, X_{j}} - s_{f, X_{j-1}} \Vert_{\ns}^2 + \Vert s_{f, X_0} \Vert_{\ns}^2 \notag \\
    & \hspace{0.5cm} \leq \lim_{n \rightarrow \infty} \sum_{j=1}^n 4C^2 h_{X_{j-1}}^{2(\vartheta'-1)\tau} + \Vert s_{f, X_0} \Vert_{\ns}^2.
\end{align}
Since the RKHS norms $\Vert s_{f, X_{j}} - s_{f, X_{j-1}} \Vert_{\ns}^2$ can be computed and $h_{X_j}$ is known,
we can perform a least-squares fit of the model $x \mapsto a + b x$ in log-space
(with $a \approx \log(2C)$ and $b \approx (\vartheta'-1)\tau$).
Due to the known smoothness $\tau$ of the kernel this leads immediately to an estimate of the power space smoothness $\vartheta$.
From a practical point of view, $X_0$ should be chosen dense enough so that $X_1, X_2, \ldots$ already allow an asymptotic behaviour of $\Vert s_{f, X_n} - s_{f, X_{n-1}} \Vert_{\ns}$.
Finally, requiring a geometric decay of the fill distance, i.e., $h_{X_n} \leq C_h \rho^n$ for $n\geq 0$,
leads to
{\small
\begin{align}
\label{eq:rkhs_upper_bound}
     \Vert f \Vert_{\ns}^2 & \leq \lim_{n \rightarrow \infty} 4C^2 C_h^{2(\vartheta'-1)\tau} \sum_{j=0}^n \rho^{2(\vartheta'-1)\tau j} + \Vert s_{f, X_0} \Vert_{\ns}^2 \notag \\
     & = \tilde{C}(1-\tilde{\rho})^{-1} + \Vert s_{f, X_0} \Vert_{\ns}^2
\end{align} }%
with $\tilde{C}=4C^2C_h^{2(\vartheta'-1)\tau}$ and $\tilde{\rho}=\rho^{2(\vartheta'-1)\tau}$,
which in turn can be combined with the fitted parameters from above to arrive at an estimated \emph{upper bound} of the RKHS norm of $f$.
The resulting method is described in \Cref{alg:alg2}.
\begin{algorithm}
\caption{RKHS norm upper bound estimation}
\label{alg:alg2}
\begin{algorithmic}[2]
\Require Nested input points $(X_i)_{i=0, ..., n_0} \subset \Omega$ with geometrically decaying fill distances, corresponding function values of $f$ 
\Ensure Estimate of upper bound on RKHS norm
\For{$i=0, ..., n_0$}
    \State Compute and store $h_{X_i}$, $\Vert s_{f, X_i} \Vert_{\ns}$
\EndFor
\State Fit model $h \mapsto c_2 h^{\beta_2}$ (determine $c_2, \beta_2$) to \\ $(h_{X_i}, \Vert s_{f, X_{i+1}} - s_{f, X_i} \Vert_{\ns})_{i=0, ..., n_0 - 1}$ \\
\Return RKHS upper bound from Eq.~\eqref{eq:rkhs_upper_bound}
\end{algorithmic}
\end{algorithm}

Note that the results describing the asymptotic convergence behaviour are rigorous and exact -- no heuristics are involved.
However, since our algorithms work on finitely many sampling points,
they are to some extent heuristic, 
though still theoretically grounded.
This situation is similar to normal approximation in statistics,
where the theoretical foundations are rigorous and exact (such as the central limit theorem),
but the practical applications are to some extent heuristic (using an asymptotic result in a non-asymptotic setting).

\begin{figure}[t]
\setlength\fwidth{.5\textwidth}
\hspace*{.52cm}\begin{tikzpicture}

\definecolor{darkgray176}{RGB}{176,176,176}
\definecolor{lightgray204}{RGB}{204,204,204}

\begin{axis}[
width=0.951\fwidth,
height=0.75\fwidth,
at={(0\fwidth,0\fwidth)},
legend cell align={left},
legend style={
  fill opacity=0.8,
  draw opacity=1,
  text opacity=1,
  at={(0.03,0.03)},
  anchor=south west,
  draw=lightgray204
},
log basis x={10},
log basis y={10},
tick align=outside,
tick pos=left,
x grid style={darkgray176},
xlabel={fill distance $h_X$},
xmin=0.008, xmax=0.2,
xmode=log,
xtick style={color=black},
y grid style={darkgray176},
ymin=3.6, ymax=4.6,
ytick style={color=black}
]
\addplot [semithick, red, mark=x, mark size=3, mark options={solid}, only marks]
table {%
0.166666666666667 3.82415899623798
0.125 3.95916375231754
0.0909090909090909 4.09264813972677
0.0645161290322581 4.20694408971636
0.0476190476190476 4.28321193822703
0.0344827586206897 4.34323517872238
0.0253164556962025 4.38512461664145
0.0185185185185185 4.41604047967603
0.0136054421768707 4.43824484773144
0.01 4.454441294808
};
\addlegendentry{\scriptsize $\Vert s_{f, X} \Vert_{\ns}^2$}
\addplot [semithick, blue]
table {%
0.00868775494941924 4.50387399563563
0.191840892885461 4.50387399563563
};
\addlegendentry{\scriptsize $\Vert f \Vert_{\ns}^2$ estimate} %
\addplot [semithick, blue, dashed]
table {%
0.166666666666667 3.76330399254343
0.165084175084175 3.77010990360685
0.163501683501683 3.77691791940795
0.161919191919192 3.78372806097898
0.1603367003367 3.79054034977002
0.158754208754209 3.79735480766146
0.157171717171717 3.80417145697705
0.155589225589226 3.81099032049742
0.154006734006734 3.81781142147418
0.152424242424242 3.82463478364458
0.150841750841751 3.83146043124685
0.149259259259259 3.83828838903607
0.147676767676768 3.84511868230084
0.146094276094276 3.85195133688062
0.144511784511785 3.8587863791838
0.142929292929293 3.86562383620664
0.141346801346801 3.87246373555294
0.13976430976431 3.87930610545474
0.138181818181818 3.88615097479388
0.136599326599327 3.89299837312453
0.135016835016835 3.89984833069688
0.133434343434343 3.90670087848181
0.131851851851852 3.91355604819688
0.13026936026936 3.92041387233347
0.128686868686869 3.92727438418533
0.127104377104377 3.93413761787845
0.125521885521886 3.94100360840254
0.123939393939394 3.94787239164406
0.122356902356902 3.9547440044209
0.120774410774411 3.96161848451897
0.119191919191919 3.96849587073066
0.117609427609428 3.97537620289535
0.116026936026936 3.98225952194218
0.114444444444444 3.9891458699351
0.112861952861953 3.99603529012048
0.111279461279461 4.00292782697743
0.10969696969697 4.00982352627092
0.108114478114478 4.01672243510808
0.106531986531987 4.02362460199781
0.104949494949495 4.0305300769139
0.103367003367003 4.03743891136207
0.101784511784512 4.04435115845111
0.10020202020202 4.05126687296847
0.0986195286195286 4.05818611146067
0.097037037037037 4.065108932319
0.0954545454545455 4.07203539587066
0.0938720538720539 4.07896556447626
0.0922895622895623 4.0858995026337
0.0907070707070707 4.09283727708941
0.0891245791245791 4.09977895695733
0.0875420875420875 4.10672461384646
0.085959595959596 4.11367432199767
0.0843771043771044 4.12062815843072
0.0827946127946128 4.12758620310235
0.0812121212121212 4.13454853907662
0.0796296296296296 4.14151525270859
0.0780471380471381 4.14848643384277
0.0764646464646465 4.15546217602775
0.0748821548821549 4.16244257674885
0.0732996632996633 4.16942773768049
0.0717171717171717 4.17641776496064
0.0701346801346801 4.18341276948965
0.0685521885521886 4.19041286725632
0.066969696969697 4.19741817969435
0.0653872053872054 4.20442883407274
0.0638047138047138 4.21144496392435
0.0622222222222222 4.21846670951738
0.0606397306397306 4.2254942183752
0.0590572390572391 4.23252764585097
0.0574747474747475 4.23956715576437
0.0558922558922559 4.24661292110914
0.0543097643097643 4.25366512484143
0.0527272727272727 4.26072396076096
0.0511447811447812 4.26778963449887
0.0495622895622896 4.27486236462899
0.047979797979798 4.2819423839224
0.0463973063973064 4.28902994076901
0.0448148148148148 4.29612530079503
0.0432323232323232 4.30322874871098
0.0416498316498317 4.310340590433
0.0400673400673401 4.31746115552969
0.0384848484848485 4.3245908000593
0.0369023569023569 4.33172990987834
0.0353198653198653 4.33887890452344
0.0337373737373738 4.34603824179612
0.0321548821548822 4.35320842321689
0.0305723905723906 4.36039000056437
0.028989898989899 4.36758358378333
0.0274074074074074 4.37478985063868
0.0258249158249158 4.3820095586249
0.0242424242424243 4.38924355982915
0.0226599326599327 4.3964928197232
0.0210774410774411 4.40375844127192
0.0194949494949495 4.4110416963794
0.0179124579124579 4.41834406769033
0.0163299663299664 4.42566730538703
0.0147474747474748 4.43301350636267
0.0131649831649832 4.44038522798704
0.0115824915824916 4.44778565767085
0.01 4.45521887723417
};
\addlegendentry{\scriptsize $4.504 - 4.194 \cdot h_{X}^{0.968}$}
\end{axis}

\end{tikzpicture}
\begin{tikzpicture}

\definecolor{darkgray176}{RGB}{176,176,176}
\definecolor{lightgray204}{RGB}{204,204,204}

\begin{axis}[
width=0.951\fwidth,
height=0.75\fwidth,
at={(0\fwidth,0\fwidth)},
legend cell align={left},
legend style={
  fill opacity=0.8,
  draw opacity=1,
  text opacity=1,
  at={(0.97,0.03)},
  anchor=south east,
  draw=lightgray204
},
log basis x={10},
log basis y={10},
tick align=outside,
tick pos=left,
x grid style={darkgray176},
xlabel={fill distance $h_X$},
xmin=0.008, xmax=0.15,
xmode=log,
xtick style={color=black},
y grid style={darkgray176},
ymin=0.3, ymax=1.5,
ymode=log,
ytick style={color=black}
]
\addplot [semithick, red, mark=x, mark size=3, mark options={solid}, only marks]
table {%
0.125 1.02507833318103
0.0909090909090909 1.03672135983779
0.0645161290322581 0.973966004740607
0.0476190476190476 0.804689961271504
0.0344827586206897 0.719574394987134
0.0253164556962025 0.604670228953936
0.0185185185185185 0.521627850735037
0.0136054421768707 0.443400282225249
0.01 0.379512740861311
};
\addlegendentry{ \scriptsize $\Vert s_{f, X_{n+1}} - s_{f, X_n} \Vert_{\ns}$}
\addplot [semithick, blue, dashed]
table {%
0.125 1.36598073132119
0.0909090909090909 1.16576498060659
0.0645161290322581 0.982842026601827
0.0476190476190476 0.844973251857921
0.0344827586206897 0.719574394987134
0.0253164556962025 0.616999062808801
0.0185185185185185 0.528078146043909
0.0136054421768707 0.452959169176053
0.01 0.388606409423971
};
\addlegendentry{\scriptsize $3.845 \cdot h_{X}^{0.498}$}
\end{axis}

\end{tikzpicture}
\caption{Exemplary visualization of the proposed algorithms.
In this example, 
the predictions for the RKHS norm $\Vert f \Vert_{\ns}$ are given by 4.504 (Algorithm 1) and 4.667 (Algorithm 2).
The numerical convergence rates are close to what is to be expected ($0.968$ vs $1.0$, $0.498$ vs $0.5$).
}
\label{fig:vis_alg}
\end{figure}
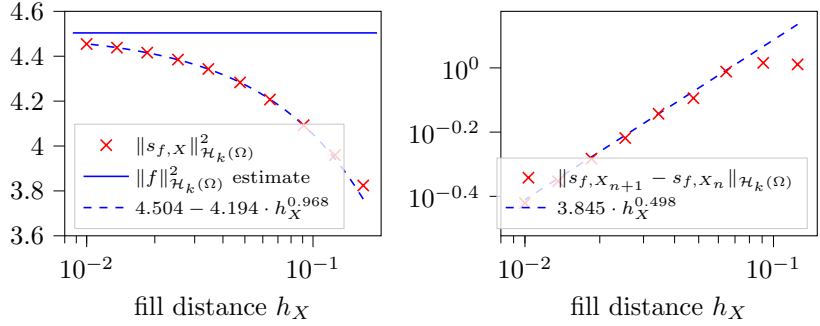
\section{Numerical experiments} \label{sec:experiments}

We now evaluate our methods with numerical experiments\footnote{Code will be made available upon acceptance}.
For concreteness, we use Matérn kernels in the following experiments, but we stress that the theory and our method are also applicable to other classes of kernels, such as the popular compactly supported Wendland kernels \cite{wendland2004scattered}.

\subsection{Test functions} \label{subsec:num_results}
Since our algorithms are meant for situations where the RKHS norm of a function is not known, it is important to evaluate whether the resulting RKHS norm estimates are reliable.
For this reason, we first apply our methods on synthetic test functions, including some for which the RKHS norm can be computed analytically.
We consider six functions on $[-1, 1]$ as well as three functions on $[0, 1]^2$, cf. Table \ref{tab:rkhs_norms}.
The univariate function $f_1$ is defined as $f(x)=\exp(x-1) (x-3) - \exp(-1-x) (x+3) + 4$,
and it is included in $T_k L_2(\Omega)$ if $k$ is the linear Matérn kernel, hence a faster convergence rate is expected.
In the multivariate setting, the function $f_2$ is given by $f(x_1,x_2)=\sin(\pi x_1) \sin(\pi x_2)$,
and $f_3$ is the well-known Franke test function.
We applied both Algorithm 1 and Algorithm 2 to these test functions,
using up to $513$ points (in 1D) respective $81^2$ points (in 2D).

Before turning to the RKHS norm predictions, we first inspect the fitted exponents $\beta_1$ and $\beta_2$, cf. \Cref{tab:fitted_coefficients}, since the predictions rely on these parameters.
Most of these are close to 1 (Algorithm 1) and 0.5 (Algorithm 2), respectively, as is expected based on the discussion in \Cref{rem:power_spaces}.
In some cases, the exponents in \Cref{tab:fitted_coefficients} are significantly higher (in \textbf{bold}), indicating superconvergence. 
For $f_1$, this behaviour is indeed expected from theory. 

Let us turn to the RKHS norm predictions by Algorithm 1 and Algorithm 2, which are reported in \Cref{tab:rkhs_norms} together with the actual RKHS norms for the univariate functions, which can be obtained analytically.
We observe that the predictions of Algorithm 1 and Algorithm 2 are always close together.
Note that the predictions from Algorithm 2 are always slightly higher than the predictions from Algorithm 1,
which is expected and reasonable since Algorithm 2 aims to estimate an \emph{upper bound} on the RKHS norm.
Furthermore, for the univariate functions for which the true RKHS norms are known, we observe that both Algorithm 1 and Algorithm 2 are rather accurate.
While the predictions of Algorithm 1 are in general slightly more accurate than those of Algorithm 2,
they sometimes slightly underestimate the true RKHS norm.
In contrast, Algorithm 2 always provides an upper bound on the true RKHS norm.

\setlength{\tabcolsep}{5pt}
\begin{table*}[t]
\centering
\resizebox{\textwidth}{!}{
 \begin{tabular}{||c | c c c | c c c | c c c ||} 
 \hline
 Function 		& \multicolumn{3}{c|}{Matérn 0} & \multicolumn{3}{c|}{Matérn 1} & \multicolumn{3}{c|}{Matérn 2}  \\ [0.5ex] \hline\hline
 				& Algorithm 1 	& Algorithm 2 & true	& Algorithm 1 	& Algorithm 2 & true	& Algorithm 1 	& Algorithm 2 & true 	\\ [0.5ex] \hline\hline
$|x|$ 			& 1.52753 		& 1.52757 	 & 1.52752 		& - 			& - 		 & $\infty$  & -            & -         & $\infty$		\\
$x^2$ 			& 1.59172 		& 1.59188    & 1.59164 		& 3.23019 		& 3.23043    & 3.23006 	 & 3.59227 		& 3.59253   & 3.59165 		\\
$\exp(-0.5x)$ 	& 1.73555 		& 1.73556 	 & 1.73553	    & 1.85329 		& 1.85332    & 1.85326	 & 1.39103 		& 1.39106   & 1.39097	    \\
$(x^2 - 1)^2$ 	& 1.27491 		& 1.27497    & 1.27491		& 2.79939 		& 2.80002    & 2.79682	 & 7.30864 		& 7.31385   & 7.30122		\\
$\sin(2 \pi x)$ & 4.49859 		& 4.49873    & 4.49880		& 21.18989 	    & 21.28833   & 21.19207	 & 64.73353 	& 65.03068  & 64.68946 	    \\
$f_1$ 		    & 2.25658 		& 2.25777    & 2.25659	    & 1.83122 		& 1.83122    & 1.83121   & 1.06552  	& 1.06552   & 1.06552	    \\ \hline

$f_2$ 			    & 2.30738		&  2.32439  & -	      & 8.96719 		& 9.14976   & -  	 & 22.27343 	 & 22.32800    & -\\
$x_1^2 + x_2^2 + 1$	& 3.54707 		& 3.56540	& -	      & 4.98415 		& 5.01081   & -		 & 4.95890 		 & 4.98154     & -\\
$f_3$               & 2.00531 		& 2.01181 	& -	      &  12.54689 		& 12.94997 	& -	     & 92.24587	     & 93.14059    & - \\ \hline
 \end{tabular}
 }
 \caption{Overview of the predicted RKHS norms $\Vert f \Vert_{\ns}$ for the two algorithms as well as some true RKHS norms, which could be calculated analytically for the univariate functions.}
 \label{tab:rkhs_norms}
\end{table*}
\setlength{\tabcolsep}{5pt}
\begin{table}[h!]
\centering
 \begin{tabular}{||c | c c | c c | c c ||} 
 \hline
 Function 		& \multicolumn{2}{c|}{Matérn 0} & \multicolumn{2}{c|}{Matérn 1} & \multicolumn{2}{c|}{Matérn 2}  \\ [0.5ex] \hline\hline
 				& Alg.\ 1 	& Alg.\ 2 	& Alg.\ 1 	& Alg.\ 2 	& Alg.\ 1 	& Alg.\ 2 	\\ [0.5ex] \hline\hline
$|x|$ 			& 0.999 		& 0.496 		& - 			& - 			& - 			& - 			\\
$x^2$ 			& 0.993 		& 0.489 		& 0.993 		& 0.490		    & 0.983 		& 0.485 		\\
$\exp(-0.5x)$ 	& 0.996 		& 0.496 		& 0.993 		& 0.494 		& 0.990 		& 0.492 		\\
$(x^2 - 1)^2$ 	& \textbf{2.129} 		& \textbf{1.064} 		& 0.912 		& 0.446 		& 0.957 		& 0.459 		\\
$\sin(2 \pi x)$ & 1.015 		& 0.503 		& 1.322 		& 0.660 		& 0.891 		& 0.441 	\\
$f_1$ 		    & 1.027 		& 0.513 		& \textbf{4.464} 		& \textbf{2.218} 		& \textbf{2.948} 		& \textbf{1.475} 		\\ \hline

$f_2$ 			& 0.897 			& 0.403 		& 0.858 			& 0.372 			& 1.484 			& 0.683 \\
$x_1^2 + x_2^2 + 1$		& 0.918			& 0.427 		& 0.936			& 0.433 			& 0.943 			& 0.439 \\
$f_3$					& \textbf{1.923}			& \textbf{0.866} 		& \textbf{2.504}			& \textbf{0.997} 			& \textbf{4.471} 			& \textbf{1.609} \\ \hline
\hline
 \end{tabular}
 \caption{Overview of the fitted exponents $\beta_1$ and $\beta_2$ for the two algorithms.
 The bold highlighted numbers are larger than expected -- for these functions, additional superconvergence is observed as elaborated in the main text.}
\label{tab:fitted_coefficients}
\end{table}

As discussed above, a major application of RKHS norm estimates are error bounds like \eqref{eq:errorBoundWithPowerFunction}.
While both $P_X(x)$ for all $x\in\Omega$ and $\Vert s_{f, X} \Vert_{\ns}$ can be explicitly computed, the RKHS norm $\|f\|_{\ns}$ in \eqref{eq:errorBoundWithPowerFunction} needs to be replaced by an estimated upper bound.
This is illustrated \Cref{fig:vis_power_func}, where two of the test functions and corresponding interpolants (using for illustrative purposes just four data points) are plotted together with the error bounds resulting from the estimates from Algorithm 2.
Note that the error bounds indeed hold in the whole domain, and we verified this additionally on a fine grid.
The second example in \Cref{fig:vis_power_func} also illustrates that these error bounds can be rather tight, which is to be expected from the worst-case optimality of \eqref{eq:errorBoundWithPowerFunction} and the accuracy of the predictions of our algorithms.
\begin{figure}[h!!]
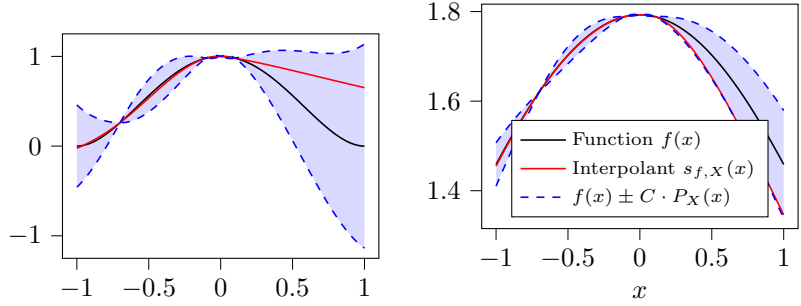

\setlength\fwidth{.5\textwidth}
\input{Figures/vis_power_bound_9_v2.tex}
\hspace*{.27cm}\input{Figures/vis_power_bound_15_v2.tex}
\caption{
Two exemplary visualization of a target function (black), a kernel interpolant (red) as well as the error bounds due to Eq.\eqref{eq:errorBoundWithPowerFunction} (blue, dashed), where $C$ is our estimate for the respective RKHS norm.
}
\label{fig:vis_power_func}
\end{figure}

\subsection{Control example}
We now apply our approach to a concrete example from learning-based control.
We consider approximate model predictive control (AMPC) with guarantees via an input-robust MPC scheme as developed in \cite{hertneck2018learning,nubert2020safe,tokmak2025automatic}.
Consider a discrete-time system
\begin{equation}
    x_+ = f(x,u)
\end{equation}
with nonlinear transition function $f\colon \R^n\times\R^m\rightarrow\R^n$ and subject to constraints $(x,u)\in \mathbb{X}\times\mathbb{U}\subseteq \R^n\times\R^m$,
together with an MPC controller $\mu\colon \R^n\rightarrow\R^m$.
Applying $\mu$ involves solving an optimal control problem, specifically a nonlinear program, in each time step, which can pose problems with real-time feasibility.
In AMPC, $\mu$ is replaced by an explicit map $\hat\mu$ that is cheap and fast to evaluate, e.g., by representing $\hat\mu$ as a neural network or a kernel interpolator.
In order to maintain control-theoretic guarantees, in particular, constraint satisfaction and recursive feasibility, one can use for $\mu$ an input-robust MPC scheme.
Formally, for a given approximation level $\epsilon>0$ one can design $\mu$ and a tightened constraint set $\hat{\mathbb{X}}_\epsilon\subseteq\mathbb{X}$ such that if 
\begin{equation} \label{eq:uniformBoundRMPC}
    \|\mu(x)-\hat\mu(x)\|\leq \epsilon
    \quad
    \forall x \in \hat{\mathbb{X}}_\epsilon
\end{equation}
then $\hat\mu$ maintains the system-theoretic guarantees of $\mu$ on $\hat{\mathbb{X}}_\epsilon$; see for example the robust MPC formulation in~\cite{kohler2020computationally}.
While \cite{hertneck2018learning,nubert2020safe} proposed to use deep learning together with an iterative post-training verification procedure,
\cite{tokmak2025automatic} suggested to use kernel interpolation together with the classic error bound \eqref{eq:errorBoundWithPowerFunction} to ensure that \eqref{eq:uniformBoundRMPC} holds, which in turn requires an upper bound on the RKHS norm of $\mu$.

We will now use the method outlined in Section \ref{sec:rkhs_norm_estimation} for this task, demonstrating it for concreteness on the nonlinear robust MPC of the continuously stirred tank reactor (CSTR) example used in \cite{tokmak2025automatic}.
The corresponding feedback map $\mu$ on its domain of feasibility $\hat{\mathbb{X}}_\epsilon$ is shown in Figure \ref{fig:cstr_mpc_plot}, where we used $\epsilon=5.1\cdot 10^{-3}$ as in \cite{tokmak2025automatic}.
\begin{figure}[h!]
    \centering
    \includegraphics[width=\linewidth]{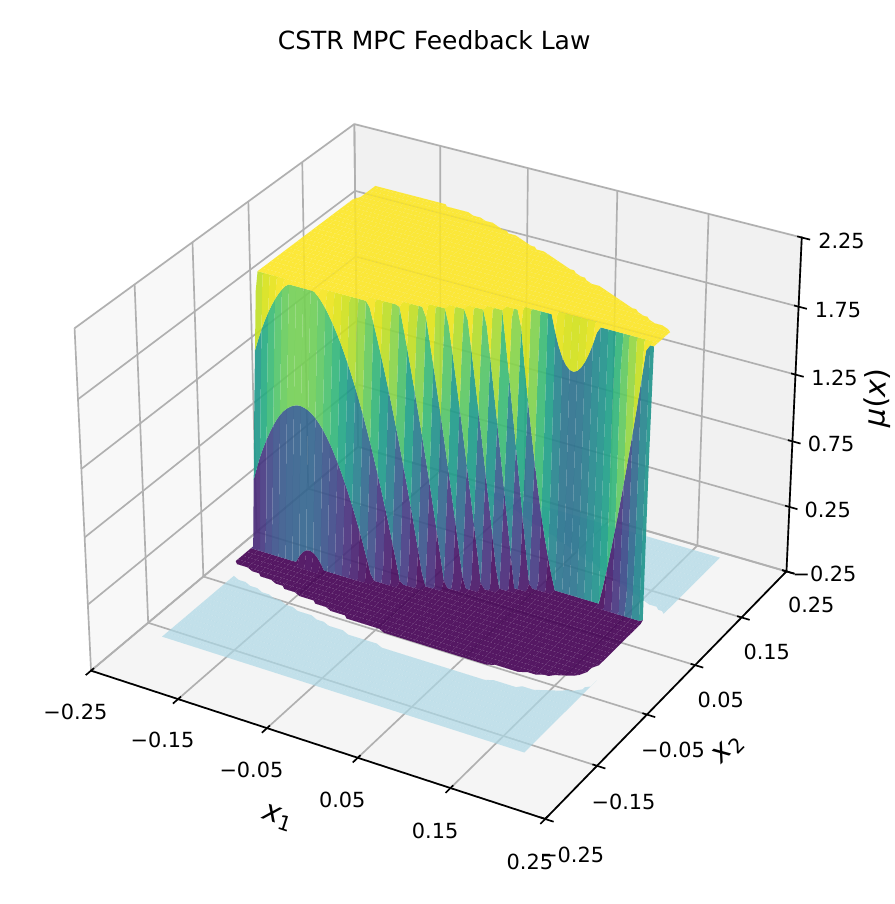}
    \caption{MPC map $\mu$ for the CSTR example. Note that $\mu$ is not defined on the light-blue area in the $x_1$-$x_2$-plane.}
    \label{fig:cstr_mpc_plot}
\end{figure}
Since the map appears to be not very smooth, we use the Matérn 0 kernel for this task with length scale 0.2 (due to the smaller domain).
Running Algorithm \ref{alg:alg1} results in an RKHS norm estimate of 45.9622. %
We evaluate the quality of this estimate by using it together with the classic error bound \eqref{eq:errorBoundWithPowerFunction}.
For this, we randomly sample 50 points from $\hat{\mathbb{X}}_\epsilon$ and fit a kernel interpolator $\hat{\mu}$, cf. \Cref{fig:cstr_error_check} (top).
We then check on a fine grid on the function domain whether the error bound $|\mu(x)-\hat{\mu}(x)|\leq C P_X(x)$ holds, where $C=45.9622$ is the RKHS norm estimate from above.
Using a grid with $100^2$ points, we did not find a single violation of the error bound, cf. \Cref{fig:cstr_error_check} (bottom), indicating that the RKHS norm estimate is indeed reliable and can be used in error bounds.
Note that for illustrative purposes, we estimated an RKHS norm for the full map $\mu$.
As the plot in Figure \ref{fig:cstr_mpc_plot} clearly shows, this function has regions with vastly different smoothness levels, and hence in practical applications, a localisation approach as proposed in \cite{tokmak2025automatic} is beneficial.
Our algorithms are directly applicable in such a setup by simply restricting them to local domains.
\begin{figure}[h!]
    \centering
    \includegraphics[width=\linewidth]{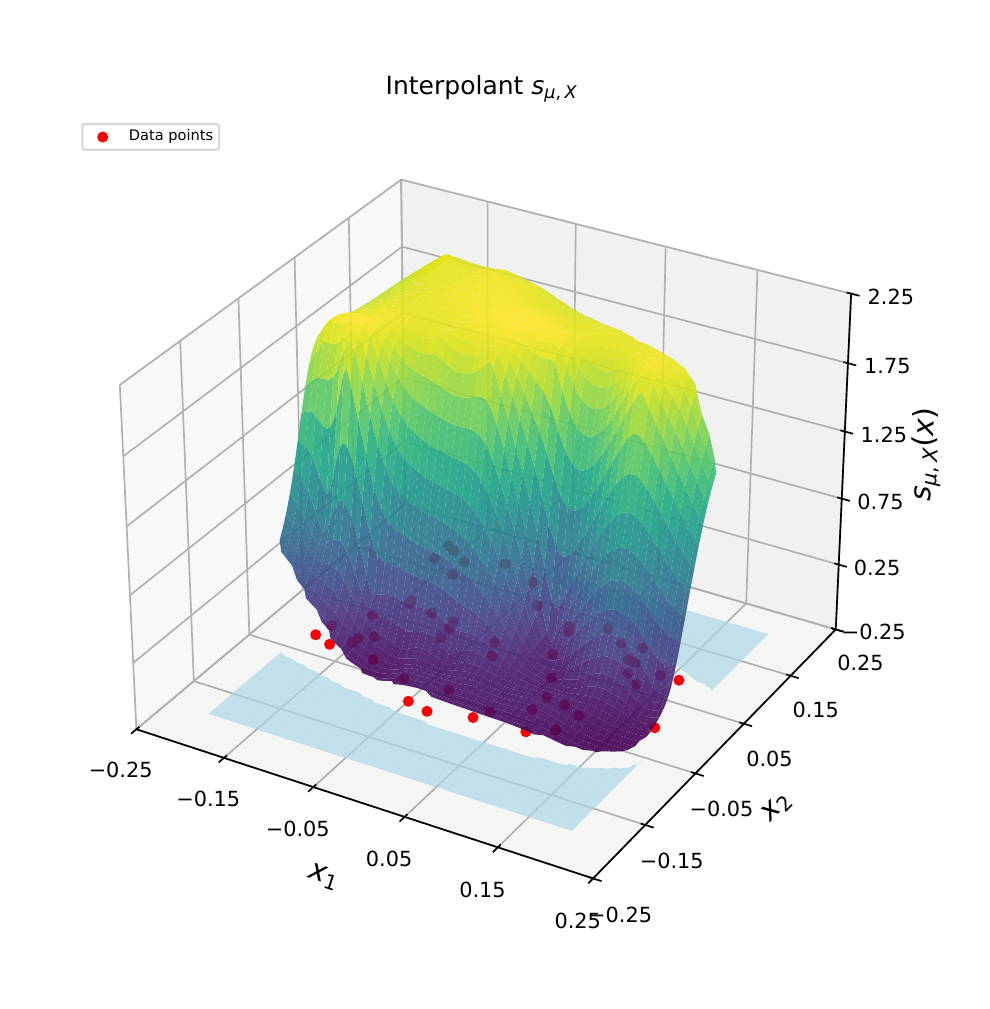}
    \includegraphics[width=0.8\linewidth]{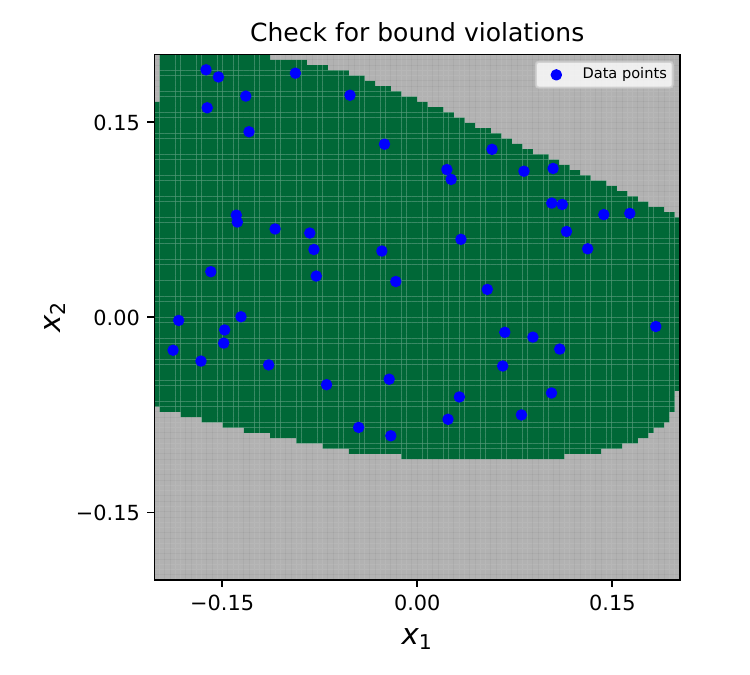}
    \caption{Kernel interpolator from 50 data points (top), verification of the error bound from the RKHS norm estimate on a fine grid (bottom).}
    \label{fig:cstr_error_check}
\end{figure}

\section{Conclusion and outlook} \label{sec:conclusion}
We considered the important problem of estimating RKHS norms of functions, which poses a major obstacle to the practical applicability of kernel methods with guarantees in learning-based control.
We tackled this challenge by proposing two sampling-based algorithms that compute from noise-free function evaluations an estimate of the RKHS norm of a target function and an upper bound thereof, respectively.
Our algorithms are based on sharp convergence results for kernel interpolators with finite smoothness.
To the best of our knowledge, our approach is the first method for RKHS norm estimation with a solid theoretical foundation.
Extensive numerical experiments and an application to a learning-based control application demonstrated the reliability and efficacy of our algorithms.

This work opens up many avenues for further research.
In this work, we assume safe and noise-free access to the target function, and extending our approach to noisy or restricted function evaluations is an interesting direction for future work.
Furthermore, since our methods require a dense grid on the input domain of the target function, which in turn suffers from the curse-of-dimensionality, our algorithms are so far limited to low-dimensional domains in practice.
Ongoing work is concerned with developing algorithms for higher-dimensional input domains based on the $P$-greedy approach \cite{santin2017convergence,wenzel2021novel}.
Finally, in this work we focused on finitely-smooth kernels due to their significance in practical applications and the availability of sharp convergence results.
Developing similar results for infinitely smooth kernels like the Gaussian kernel and adapting our algorithms to this setting is another very relevant direction for future work.

\section*{AI Statement}
Programming of the experiments was partially supported by Claude Code \cite{claudeCodeRef} and GitHub CoPilot \cite{githubCopilotRef}.
ChatGPT \cite{chatGPTref} was used to improve formulations and to the check the manuscript for typos and language problems.

\section*{Acknowledgment}
T.W.\ gratefully acknowledges funding by Daimler and Benz Foundation as part of the scholarship program for junior professors and postdoctoral researchers.
A.T.\ was supported by a Tandem Industry Academia Seed funding from the Finnish Research Impact foundation.
C.F.\ acknowledges funding from DFG Project FO 767/10-2 (eBer-24-32734) ”Implicit Bias in Adversarial Training”.
\begin{center}
    \raisebox{0pt}{\includegraphics[height=1.5cm]{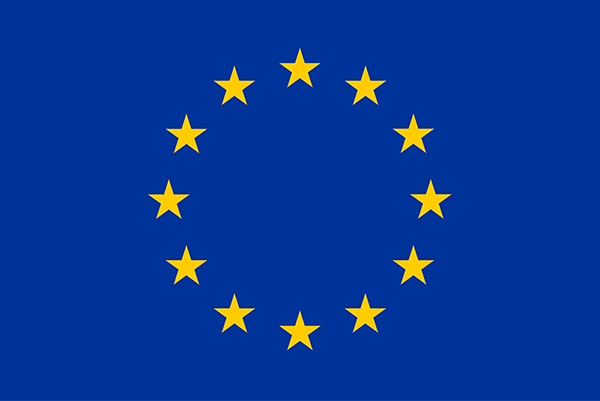}}%
    \hspace{1em}%
    \raisebox{0pt}{\includegraphics[height=1.5cm]{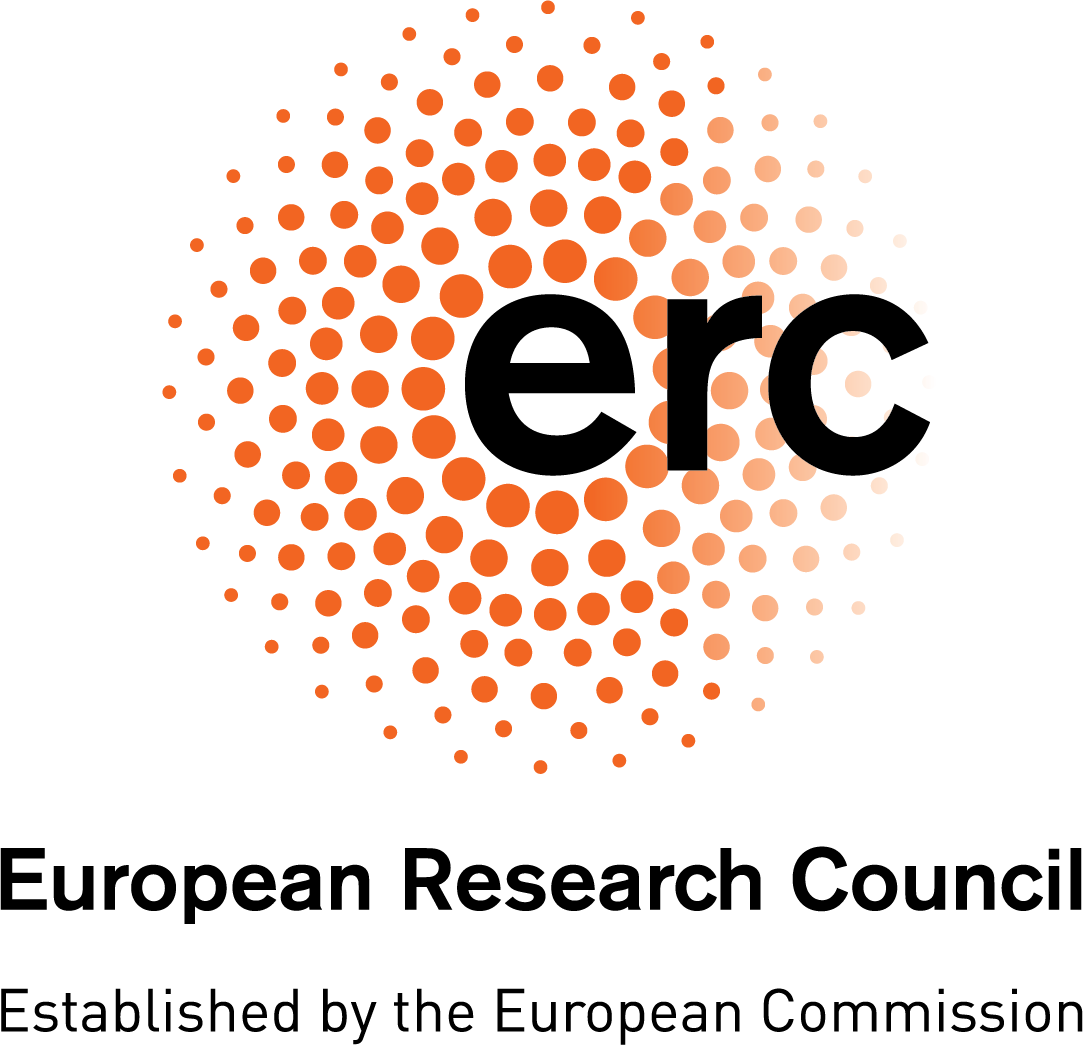}}%
\end{center}
Funded by the European Union. Views and opinions expressed are however those of the author(s) only and do not necessarily reflect those of the European Union or the European Research Council Executive Agency. Neither the European Union nor the granting authority can be held responsible for them. This project has received funding from the European Research Council (ERC) under the European Union’s Horizon Europe research and innovation programme (grant agreement No. 101198055, project acronym NEITALG).

\bibliography{../references}
\bibliographystyle{abbrv}

\end{document}